\documentclass[11pt,a4paper,epsf]{article}
\usepackage{latexsym}
\usepackage[utf8]{inputenc}
\usepackage{amsmath}
\usepackage{amsthm}
\usepackage{amsfonts}
\usepackage{amssymb}
\usepackage{epsfig}
\usepackage{nicefrac}
\usepackage[all]{xy}
\usepackage{graphicx}
\usepackage{enumerate}
\newtheorem{theorem}{Theorem}[section]
\newtheorem{corollary}[theorem]{Corollary}
\newtheorem{proposition}[theorem]{Proposition}
\newtheorem{lemma}[theorem]{Lemma}

\newtheorem{conjecture}[theorem]{Conjecture}

\theoremstyle{definition}

\newtheorem{remark}[theorem]{Remark}

\newtheorem{definition}[theorem]{Definition}

\newcommand{\CC}{{\mathbb C}}

\newcommand{\RR}{{\mathbb R}}

\begin{document}

\title{Some optimal embeddings of symplectic ellipsoids}

\author{R. Hind}

\date{\today}

\maketitle

\begin{abstract}
We construct symplectic embeddings in dimension $2n \ge 6$ of ellipsoids into the product of a $4$-ball or $4$-dimensional cube with Euclidean space. A sequence of the embeddings are seen to be optimal by a quantitative version of the embedding obstructions introduced in \cite{hindker}.

In the limiting case when our ellipsoids approach a cylinder we recover an embedding of Guth, \cite{guth}. However for compact ellipsoids our embedding gives sharper results. At the other end of the scale, for certain convergent sequences of ellipsoids we reproduce estimates given by stabilizing $4$-dimensional embeddings of McDuff and Schlenk, \cite{mcdsch}, in the ball case and  Frenkel and M\"{u}ller, \cite{frmu}, in the cube case.

\end{abstract}

\begin{section}{Introduction}

Recently there has been much progress on the problem of symplectic embeddings, particularly in dimension $4$. McDuff and Schlenk in \cite{mcdsch} completely classified ellipsoid embeddings into balls, and Frenkel and M\"{u}ller in \cite{frmu} classified ellipsoid embeddings into cubes. Meanwhile Hutchings has developed the Embedded Contact Homoloy capacities, see \cite{hutchings}, which provide a complete set of invariants in the cases mentioned above, see \cite{mcd2}, \cite{hutchings2}, together with many others, see \cite{choi}.

In higher dimensions much less is known. In the present paper we construct a new embedding of ellipsoids using a version of multiple symplectic folding, see \cite{schl}. Furthermore, results from \cite{hindker} show that a sequence of such embeddings are optimal. This is slightly surprising since in dimension $4$, at least in the case of ellipsoid embeddings into a ball, we see from \cite{mcdsch} and \cite{schl} that folding never gives optimal embeddings (except in the trivial case when the inclusion is optimal).

We can compare our results with known constructions.
In the limiting case when the ellipsoids approach a cylinder (becoming infinitely thin) we recover an optimal embedding result of Guth, \cite{guth}, which is related to the nonexistence of higher order embedding capacities, but in the compact case our embedding gives strictly sharper results, see section \ref{1pt1} for a discussion of this. This is the case at least for ellipsoids, for polydisks our construction does not seem to give anything new.

On the other hand, for certain convergent sequences of rounder ellipsoids our embedding gives the same result as stabilizing $4$-dimensional embeddings of McDuff and Schlenk, \cite{mcdsch}, and  Frenkel and M\"{u}ller, \cite{frmu}. This is discussed in section \ref{capacities}. We emphasize that even though our embeddings and these stabilizations have the same domain and range, the embeddings themselves are very different. The McDuff--Schlenk and  Frenkel--M\"{u}ller embeddings are constructed indirectly using holomorphic curves and Seiberg--Witten theory. The embeddings constructed here are very concrete.

Before stating our results we fix some definitions.
We study symplectic embeddings into Euclidean space $\RR^{2n}$, with coordinates $x_j,y_j$, $1 \le j \le n$, equipped with its standard symplectic form $\omega = \sum_{j=1}^n dx_j \wedge dy_j$. Often it is convenient to identify $\RR^{2n}$ with $\CC^n$ by setting $z_j = x_j + i y_j$. The basic domains for symplectic embedding problems are ellipsoids $E$ and polydisks $P$ which we define as follows.

$$E(a_1, \dots ,a_n) = \{\sum_j \frac{\pi |z_j|^2}{a_j} \le 1\};$$
$$P(a_1, \dots ,a_n) = \{\pi |z_j|^2 \le a_j \, \mathrm{for \, all} \, j\}.$$

These are subsets of $\CC^n$ and so inherit the symplectic structure. A ball of capacity $R$ is simply an ellipsoid $B^{2n}(R) = E(R, \dots ,R)$. A disk of area $a$ will sometimes be written $D(a) = B^2(a)$. It is also convenient to write $\lambda E(a_1, \dots ,a_n)$ for $E(\lambda a_1, \dots ,\lambda a_n)$.

The following notation will be useful.

\begin{definition} \[E(a_1, \dots, a_n) \hookrightarrow E(b_1, \dots, b_n)\] will mean that for all $0 < \epsilon <1$ there exists a symplectic embedding $(1-\epsilon) E(a_1, \dots, a_n) \rightarrow E(b_1, \dots, b_n)$.
\end{definition}

\begin{remark} It is of course true that if there exists a symplectic embedding of the interior $\mathring{E}(a_1, \dots, a_n) \rightarrow E(b_1, \dots, b_n)$ then we can also write $E(a_1, \dots, a_n) \hookrightarrow E(b_1, \dots, b_n)$. In dimension $4$, that is when $n=2$, the converse is also true. This follows from a theorem of McDuff, \cite{mcduff} Corollary $1.6$, saying that the space of ellipsoid embeddings into an ellipsoid is path connected. It is unknown if the converse is true in higher dimensions.
However we can get an embedding of the interior in any dimension if there exists a family of symplectic embeddings $(1-\epsilon) E(a_1, \dots, a_n) \rightarrow E(b_1, \dots, b_n)$ depending smoothly on $\epsilon$. This is a consequence of a result of Pelayo and Ng\d{o}c, \cite{pelngo1}, Theorem $4.1$. It is likely that the main construction in the present paper can be carried out smoothly in a family, but we do not discuss that here.
\end{remark}

Given the notation above, we can state our main theorem as follows.

\begin{theorem} \label{main} Let $n \ge 3$ and $a_2 \ge 1$. Then $$E(1,a_2, \dots, a_n) \hookrightarrow B^4(\frac{3a_2}{a_2+1}) \times \CC^{n-2}.$$
Moreover, if $a_2 = 3d-1$ for a positive integer $d$ and $a_3, \dots ,a_n \ge a_2$ then the embedding is optimal, that is, if $R<\frac{3a_2}{a_2+1}$ then there are no symplectic embeddings $E(1,a_2, \dots, a_n) \rightarrow B^4(R) \times \CC^{n-2}$.
\end{theorem}

An analogous result in dimension $4$ would discuss symplectic embeddings of an ellipsoid into a $4$-ball, and this problem has been completely solved by McDuff and Schlenk in \cite{mcdsch}. Relations to the $4$-dimensional situation are discussed in section \ref{capacities}. It turns out that the first part of Theorem \ref{main}, the existence of an embedding, is also true in dimension $4$ provided $a_2 \le \tau^4 = \frac{7+3\sqrt{5}}{2}$ but is false when $a_2>\tau^4$. Here $\tau$ here is the golden ratio. The embedding is optimal when $n=2$ provided $a_2=\frac{g_{n+2}}{g_N}$, a ratio of odd Fibonacci numbers (see again section \ref{capacities} for details). This includes the cases $a_2=3d-1$ and $d=1,2$. For $d=3$ we have $E(1,8) \hookrightarrow B^4(R)$ if and only if $R \ge \frac{17}{6}$ and for $a \ge 8\frac{1}{36}$ we have $E(1,a) \hookrightarrow B^4(R)$ if and only if $R \ge \sqrt{a}$.

There is a similar statement for embeddings into products of a bidisk and Euclidean space.

\begin{theorem} \label{main2} Let $n \ge 3$ and $a_2 \ge 1$. Then $$E(1,a_2, \dots, a_n) \hookrightarrow P(\frac{2a_2}{a_2+1}, \frac{2a_2}{a_2+1}) \times \CC^{n-2}.$$
Moreover, if $a_2 = 2d+1$ for a positive integer $d$ and $a_3, \dots ,a_n \ge a_2$ then the embedding is optimal, that is, if $R<\frac{2a_2}{a_2+1}$ then there is no symplectic embedding $E(1,a_2, \dots, a_n) \rightarrow P(R,R) \times \CC^{n-2}$.
\end{theorem}

Now, by work of Frenkel and M\"{u}ller, \cite{frmu}, we see that the embedding result is true also in dimension $4$ (giving embeddings into a cube) provided $a_2 \le \sigma^2 = 3+2\sqrt{2}$, but is false for $a_2 > \sigma^2$. Here $\sigma$ is the silver ratio. The embedding is optimal if $a_2$ is a certain ratio Pell numbers which includes the cases $a_2=2d+1$ for $d=1,2$. This is discussed in section \ref{capacities}.

{\bf Outline of the paper.}

The embeddings claimed in Theorems \ref{main} and \ref{main2} are constructed in section \ref{constr}, and we describe the obstructions which imply sharpness in section \ref{obstr}. However before this we discuss relations to earlier work. In section \ref{1pt1} we show why our embedding improves estimates coming from Guth's construction. In section \ref{capacities} we describe some intriguing similarities with optimal $4$-dimensional embeddings.

\begin{subsection}{Higher order capacities and Guth's embedding.}\label{1pt1}

The study of symplectic embeddings was initiated in Gromov's seminal paper \cite{gromov}, and in particular by his nonsqueezing theorem.

\begin{theorem}\label{ns} (Gromov \cite{gromov} Corollary $0.3.A$) Suppose $a_2, \dots ,a_n \ge a_1$ and $b_2, \dots ,b_n \ge b_1$. Then $$E(a_1, \dots, a_n) \hookrightarrow E(b_1, \dots, b_n)$$ only if $a_1 \le b_1$.
\end{theorem}

We can replace either or both of the ellipsoids in Theorem \ref{ns} by polydisks and the statement still holds.

Several years later, Hofer asked in \cite{hofer} whether the size of the second factor similarly influences symplectic embeddings. In particular he asked the following.

{\bf Question.} (Hofer, \cite{hofer}, page $17$) Does there exist a symplectic embedding $$D(1) \times \CC^{n-1} := E(1, \infty, \dots, \infty) \hookrightarrow E(R,R,\infty \dots, \infty):=B^4(R) \times \CC^{n-2}$$ for any $R$?

A positive answer implies that a sort of infinite squeezing in the second factor is possible, and so a negative answer was expected. It was surprising then when an ingenious construction of Guth \cite{guth} produced embeddings at least of compact subsets of the domain ellipsoid. Guth's construction has since been quantified by Hind and Kerman in \cite{hindker} and then extended to the whole interior by Pelayo and Ng\d{o}c. The conclusion is as follows.

\begin{theorem} \label{emb1}(Pelayo-Ng\d{o}c, \cite{pelngo2} Theorem $1.2$, \cite{pelngo1} Theorem $3.3$) There exist symplectic embeddings $$\mathring{D}(1) \times \CC^{n-1} \rightarrow B^4(3) \times \CC^{n-2}$$ and $$\mathring{D}(1) \times \CC^{n-1} \rightarrow P(2,2) \times \CC^{n-2}.$$
\end{theorem}

By letting $a_3, \dots ,a_n \to \infty$ we see that our main Theorems \ref{main} and \ref{main2} recover these embeddings, at least of compact subsets of $\mathring{D}(1) \times \CC^{n-1}$.
The existence of the embeddings of Theorem \ref{emb1} imply that there are no higher order symplectic capacities as defined in \cite{hofer}, see \cite{pelngo1}.

Let us briefly recall the construction of Theorem \ref{emb1}, at least applied to large compact subsets. It relies on two lemmas. Here we let $\Sigma(\delta)$ denote a once punctured $2$-torus with a symplectic form of area $\delta$.

\begin{lemma} \label{glem1} For all large $S$ and all $\delta>0$ there exists a symplectic embedding $$\phi_1 : B^{2(n-1)}(S) \rightarrow \Sigma(\delta) \times \CC^{n-2}.$$
\end{lemma}

\begin{lemma} \label{glem2} For all $\epsilon>0$ there exists a $\delta>0$ such that we have symplectic embeddings $$\phi_2:D(1-\epsilon) \times \Sigma(\delta) \rightarrow B^4(3)$$ and $$\phi'_2:D(1-\epsilon) \times \Sigma(\delta) \rightarrow P(2,2).$$
\end{lemma}

The compositions $(\phi_2 \times \mathrm{id}) \circ (\mathrm{id} \times \phi_1)$ and $(\phi'_2 \times \mathrm{id}) \circ (\mathrm{id} \times \phi_1)$ give symplectic embeddings $D(1-\epsilon) \times B^{2(n-1)}(S) \rightarrow B^4(3) \times \CC^{n-2}$ and $D(1-\epsilon) \times B^{2(n-1)}(S) \rightarrow P(2,2) \times \CC^{n-2}$ respectively.

Now, it was shown in \cite{hindker} that Guth's construction is optimal in the following sense.

\begin{theorem} \label{gopt} (Hind-Kerman \cite{hindker} Theorem $1.2$ and Theorem $1.5$) If $R<3$ then there exist $\epsilon, S >0$ such that there does not exist a symplectic embedding $$D(1-\epsilon) \times B^{2(n-1)}(S) \rightarrow B^4(R) \times \CC^{n-2}.$$
If $R_1<2$ or $R_2<2$ then there exist $\epsilon, S >0$ such that there does not exist a symplectic embedding $$D(1-\epsilon) \times B^{2(n-1)}(S) \rightarrow P(R_1,R_2) \times \CC^{n-2}.$$
\end{theorem}

It follows that the embeddings of Lemma \ref{glem2} must also be optimal, that is, we have the following.

\begin{theorem} If $R<3$ and $\delta>0$, then for sufficiently small $\epsilon>0$ there does not exist a symplectic embedding $$D(1-\epsilon) \times \Sigma(\delta) \rightarrow B^4(R).$$ If $\min(R_1,R_2)<2$ and $\delta>0$, then for sufficiently small $\epsilon>0$ there does not exist a symplectic embedding $$D(1-\epsilon) \times \Sigma(\delta) \rightarrow P(R_1,R_2).$$
\end{theorem}

In conclusion, this kind of construction, that is, factoring through $\phi_1$, will not give symplectic embeddings even of compact ellipsoids $E(1,S, \dots ,S)$ into a $B^4(R) \times \CC^{n-2}$ with $R<3$ or a $P(R_1,R_2) \times \CC^{n-2}$ with $R_1<2$. The purpose of this paper is to show that nevertheless many such embeddings do exist.

\end{subsection}

\begin{subsection}{$4$-dimensional embeddings and capacities.}\label{capacities}

A complete understanding of our embedding problem, in the case of ellipsoid embeddings into products of a ball and Euclidean space, amounts to computing the capacity function $$f(a_2, \dots ,a_n) = \inf \{R|E(1,a_2, \dots, a_n) \hookrightarrow B^4(R) \times \CC^{n-2}\}.$$

We know very little about such functions. For example, Figure $1$ in \cite{busehind} describes our sparse knowledge of embeddings of $6$-dimensional ellipsoids into balls.

As far as obstructions are concerned (that is, lower bounds on $f$), the only known invariants in dimension $2n \ge 6$ besides those discussed in section \ref{obstr} are the Ekeland-Hofer capacities, see \cite{ekehof}. These imply the following.

\begin{proposition}\label{cap1} Suppose $1 \le a_2 \le \dots \le a_n$. If $a_2 \le 2$ then $f(a_2, \dots ,a_n)=a_2$. If $a_2 \ge 2$ then $f(a_2, \dots ,a_n) \ge 2$.
\end{proposition}

\begin{proof}
First suppose $a_2 \le 2$. Then the second Ekeland-Hofer capacity of the ellipsoid $E=E(1,a_2, \dots, a_n)$ is $c_2(E)=a_2$. The corresponding capacity of the product $Z(R)=B^4(R) \times \CC^{n-2}$ is $c_2(Z(R))=R$ and so by monotonicity of the capacities we have $f(a_2, \dots ,a_n) \ge a_2$. Since the inclusion map gives an embedding $E \to Z(a_2)$, this is in fact an equality.

Next suppose $a_2 \ge 2$. Now the second Ekeland-Hofer capacity of the ellipsoid $E$ is $c_2(E)=2$ and so we see that $f(a_2, \dots ,a_n) \ge 2$.
\end{proof}

For constructions (that is, upper bounds on $f$) we note that if we have a symplectic embedding $\phi:E(1,a_2) \rightarrow B^4(R)$ then the product $(z_1, z_2, \dots ,z_n) \mapsto (\phi(z_1,z_2), z_3, \dots ,z_n)$ of course gives an embedding $E(1,a_2, \dots, a_n) \hookrightarrow B^4(R) \times \CC^{n-2}$. Hence we see immediately that we have the bound
\begin{equation}\label{one}f(a_2, \dots ,a_n) \le c_B(a_2)\end{equation}
where $c_B$ is the analogous $4$-dimension capacity $$c_B(a) = \inf \{R|E(1,a) \hookrightarrow B^4(R) \}.$$
It turns out that the function $c_B$ has been completely worked out by McDuff and Schlenk in \cite{mcdsch}. One conclusion is that $c_B(a)=2$ when $2 \le a \le 4$. Together with Proposition \ref{cap1} this implies that equation (\ref{one}) is an equality when $2 \le a_2 \le 4$ provided $a_3, \dots ,a_n \ge a_2$. In other words we have the following.

\begin{corollary} Suppose $2 \le \min(a_2, \dots ,a_n) \le 4$. Then $f(a_2, \dots ,a_n) = 2$.
\end{corollary}

A consequence of our main theorem, together with our knowledge of the function $c_B$ from \cite{mcdsch}, is that inequality (\ref{one}) is strict when $a_2> \tau^4$ but conceivably could be an equality when $a_2 \le \tau^4$ provided $a_3, \dots ,a_n \ge a_2$.

Let us describe a part of the function $c_B$ which relates nicely to our construction. Let $g_0=1$ and $\{g_n\}_{n=0}^{\infty}$ be the sequence of odd Fibonacci numbers, that is, the sequence beginning $1, 2, 5, 13, 34, \dots$. Then we can define a sequence $\{b_n\}_{n=0}^{\infty}$ by $b_n=\frac{g_{n+2}}{g_n}$. We have $\lim_{n \to \infty} b_n = \tau^4 = \frac{7+3\sqrt{5}}{2}$. Given this, Theorem $1.1.2$ in \cite{mcdsch}, together with a little manipulation of Fibonacci numbers, implies the following.

\begin{theorem} \label{cfn} (McDuff-Schlenk, \cite{mcdsch}, Theorem $1.1.2$) For all $n \ge 0$ we have $c_B(b_n) = \frac{g_{n+2}}{g_{n+1}} = \frac{3b_n}{b_n+1}$.
\end{theorem}

We compare all of this with our main theorem, which can be stated as follows.

\begin{theorem}\label{mainn} $$f(a_2, \dots ,a_n) \le \frac{3a_2}{a_2+1}.$$ Moreover, we have equality in the case when $a_2=3d-1$ for a positive integer $d$ and $a_3, \dots ,a_n \ge a_2$.
\end{theorem}

Note that $f(a_2, \dots ,a_n)$ is clearly a nondecreasing function of $a_2$ if the remaining variables are held fixed. Therefore, assuming $a_3, \dots ,a_n \ge a_2 \ge 2$, the equality part of the statement of the theorem gives $$f(a_2, \dots ,a_n) \ge 3 - \frac{1}{\lfloor (a_2+1)/3 \rfloor}.$$ We observe that this matches the bound from Proposition \ref{cap1} when $2 \le a_2 < 5$ (which is sharp when $a_2 \le 4$), but improves that bound when $a_2 \ge 5$. For $5 \le a_2 <8$ our bound gives $f(a_2, \dots ,a_n) \ge \frac{5}{2}$. Since we also have $c_B(a)=\frac{5}{2}$ when $5 \le a \le \frac{25}{4}$ inequality (\ref{one}) implies the following.

\begin{corollary} Suppose $5 \le \min(a_2, \dots ,a_n) \le \frac{25}{4}$. Then $f(a_2, \dots ,a_n) = \frac{5}{2}$.
\end{corollary}

Next we see that if $a_2 = b_n$ for some $n$ then Theorem \ref{mainn} reproduces precisely the bound (\ref{one}) coming from Theorem \ref{cfn}. In other words, for these values of $a_2$ our folding construction gives exactly the same result as the product map above stabilizing a $4$-dimensional embedding. These embeddings are optimal if $n=0,1$ by the second part of Theorem \ref{mainn}, and so one might naturally conjecture the following.

\begin{conjecture}\label{cj} Suppose $a_3, \dots ,a_n \ge a_2=b_n$. Then the product embedding $E(1, b_n, a_3, \dots, a_n) \hookrightarrow B^4(\frac{g_{n+2}}{g_{n+1}}) \times \CC^{n-2}$ is optimal, that is, $f(b_n, a_3, \dots ,a_n)=c_B(b_n)$.
\end{conjecture}

In general, examining the McDuff-Schlenk function $c_B(a)$ in more detail, we see that the product embedding gives a better result than the construction of Theorem \ref{mainn} when $a_2 < \tau^4$ and $a_2 \neq b_n$ for all $n$. Our construction strictly improves on the product map when $a_2 > \tau^4$.

There is a similar story for embeddings into products of a cube and Euclidean space. Define $$g(a_2, \dots ,a_n) = \inf \{R|E(1,a_2, \dots, a_n) \hookrightarrow P(R,R) \times \CC^{n-2}\}.$$ Then Theorem \ref{main2} can be stated as follows.

\begin{theorem}\label{mainn2} $$g(a_2, \dots ,a_n) \le \frac{2a_2}{a_2+1}.$$ Moreover, we have equality in the case when $a_2=2d+1$ for a positive integer $d$ and $a_3, \dots ,a_n \ge a_2$.
\end{theorem}
A consequence is that if $a_3, \dots ,a_n \ge a_2 \ge 3$ then $$g(a_2, \dots ,a_n) \ge 2 - \frac{1}{\lfloor (a_2-1)/2 \rfloor +1}.$$
In this case the Ekeland-Hofer capacities imply only that $g(a_2, \dots ,a_n) \ge \min(1,a_2, \dots ,a_n)$, which we know already from Gromov's Theorem \ref{ns}.

The corresponding $4$-dimensional embedding capacity is $$c_P(a) = \inf \{R|E(1,a) \hookrightarrow P(R,R) \}$$ and this was worked out by Frenkel and M\"{u}ller in \cite{frmu}.

To describe a connection to Theorem \ref{mainn2} we must introduce the sequences of Pell numbers $\{P_n\}_{n=0}^{\infty}$ and half companion Pell numbers $\{H_n\}_{n=0}^{\infty}$. These are defined by the recursion relations $$P_0=0, P_1=1, \quad P_n=2P_{n-1}+P_{n-2},$$ $$H_0=1, H_1=1, \quad H_n=2H_{n-1}+H_{n-2}.$$ Next we define a sequence $\{\beta_n\}_{n=0}^{\infty}$ by
\[\beta_n = \left\{ \begin{array}{ll} \frac{H_{n+2}}{H_n} & \mbox{if $n$ is even} \\
\frac{P_{n+2}}{P_n} & \mbox{if $n$ is odd}
\end{array}
\right. \]
We have $\lim_{n \to \infty} \beta_n = \sigma^2$. A part of Theorem $1.3$ of \cite{frmu}, together with some manipulations of the recursion relations for Pell and half companion Pell numbers, gives the following.

\begin{theorem}\label{cfn2}(Frenkel-M\"{u}ller, \cite{frmu}, Theorem $1.3$) For all $n \ge 0$, $c_P(\beta_n) = \frac{2\beta_n}{\beta_n+1}$.
\end{theorem}

Hence when $a_2=\beta_n$ our embedding result matches the embedding given by stabilizing a $4$-dimensional embedding. By the second statement of Theorem \ref{mainn2} these embeddings are optimal if $n=0,1$ and it is natural to conjecture that they are optimal for all $n$.

\begin{conjecture}\label{cj2} Suppose $a_3, \dots ,a_n \ge a_2=\beta_n$. Then the product embedding $E(1, \beta_n, a_3, \dots, a_n) \hookrightarrow P(\frac{2\beta_n}{\beta_n+1}, \frac{2\beta_n}{\beta_n+1}) \times \CC^{n-2}$ is optimal, that is $g(\beta_n, a_3, \dots ,a_n)=c_P(\beta_n)$.
\end{conjecture}

In general, we see that the product embedding gives a better result than the construction of Theorem \ref{mainn2} when $a_2 < \sigma^2$ and $a_2 \neq \beta_n$ for all $n$. Our construction strictly improves on the product map when $a_2 > \sigma^2$.

\end{subsection}

\end{section}

\begin{section}{Main construction}\label{constr}

Let $S,T \ge 1$. Reordering the first and second factors the goal of this section is to produce an embedding $$E(S,1,T, \dots ,T) \hookrightarrow \left( B^4 \left( \frac{3S}{S+1} \right) \cap P \left( \frac{2S}{S+1}, \frac{2S}{S+1} \right) \right) \times \CC^{n-2}.$$ In fact, choosing an embedding to fix the $z_4, \dots ,z_n$ coordinates, it suffices to work in dimension $6$ and construct an embedding of the ellipsoid $E(S,1,T)$. By definition, we need to symplectically embed $E(S,1,T)$ in an arbitrarily small neighborhood of $\left( B^4 \left( \frac{3S}{S+1} \right) \cap P \left( \frac{2S}{S+1}, \frac{2S}{S+1} \right) \right) \times \CC$. To simplify things then, when we write an embedding we will always mean an embedding into an arbitrarily small neighborhood, and subsets will mean subsets of arbitrarily small neighborhoods, even if we do not say so explicitly. Similarly we will let $\epsilon$ denote a parameter which can be chosen to be arbitrarily small.

We begin by recalling some terminology for Hamiltonian diffeomorphisms. A compactly supported function $H:[0,1] \times \CC^n \to \RR$ will be called a Hamiltonian and sometimes written as $(t,z) \mapsto H^t(z)$. We can define the norm of $H$ by $|H| = \int_0^1 (\max H^t - \min H^t) dt$. For example, in the case when $H$ is independent of time (which will usually be the case for us) we have $|H| = \max H - \min H$. Associated to $H$ is a time dependent vector field $i \nabla H^t$ where $\nabla H^t$ denotes the usual gradient of a function on Euclidean space. The corresponding flow exists for all time since $H$ is compactly supported and we write $\phi^t$ for the time $t$ flow. Then $\phi = \phi^1$ is called the Hamiltonian diffeomorphism generated by $H$. The Hofer norm of $\phi$ is the infimum of $|H|$ over all $H$ which generate $\phi$.

{\bf Notation.}

We set $\lambda = \frac{S}{S+1}$. As $S \ge 1$ we have $\frac{1}{2} \le \lambda \le 1$. Then let $V$ be the subset of the $z_1$-plane given by $$V = \left( [0,1] \times [0,\lambda] \right) \bigcup \left( [1,2N+1] \times \{0\} \right).$$ Here $N$ is an integer of order $\frac{S}{\epsilon}$. There exists a symplectomorphism $\psi$ from $D(S)$, the disk of area $S$, to an $\epsilon$-neighborhood of $V$. We can choose this symplectomorphism such that the disk $D(\lambda) \subset D(S)$ is mapped to a neighborhood of the square $[0,1] \times [0,\lambda]$.

Let $D_1=D(\lambda)$ and $D_2$ be another disk of area $\lambda$ lying in the annulus $D(2(\lambda+\epsilon)) \setminus D(\lambda + \epsilon)$. We think of $D_1$ and $D_2$ as subsets of the $z_2$-plane.

Next let $A_i$ denote the annulus $A_i = D(i(T+\epsilon)) \setminus D((i-1)(T+\epsilon))$, which we think of as a subset of the $z_3$-plane. Let $B_1=D(T)$ and in general let $B_i$ be a disk of area $T$ lying inside $A_i$.

Now we can define the bidisks $P_i$ in the $(z_2,z_3)$-plane by $P_i = D_1 \times B_i$ if $i$ is odd, and $P_i = D_2 \times B_i$ if $i$ is even.

\begin{figure}
\centerline{\includegraphics{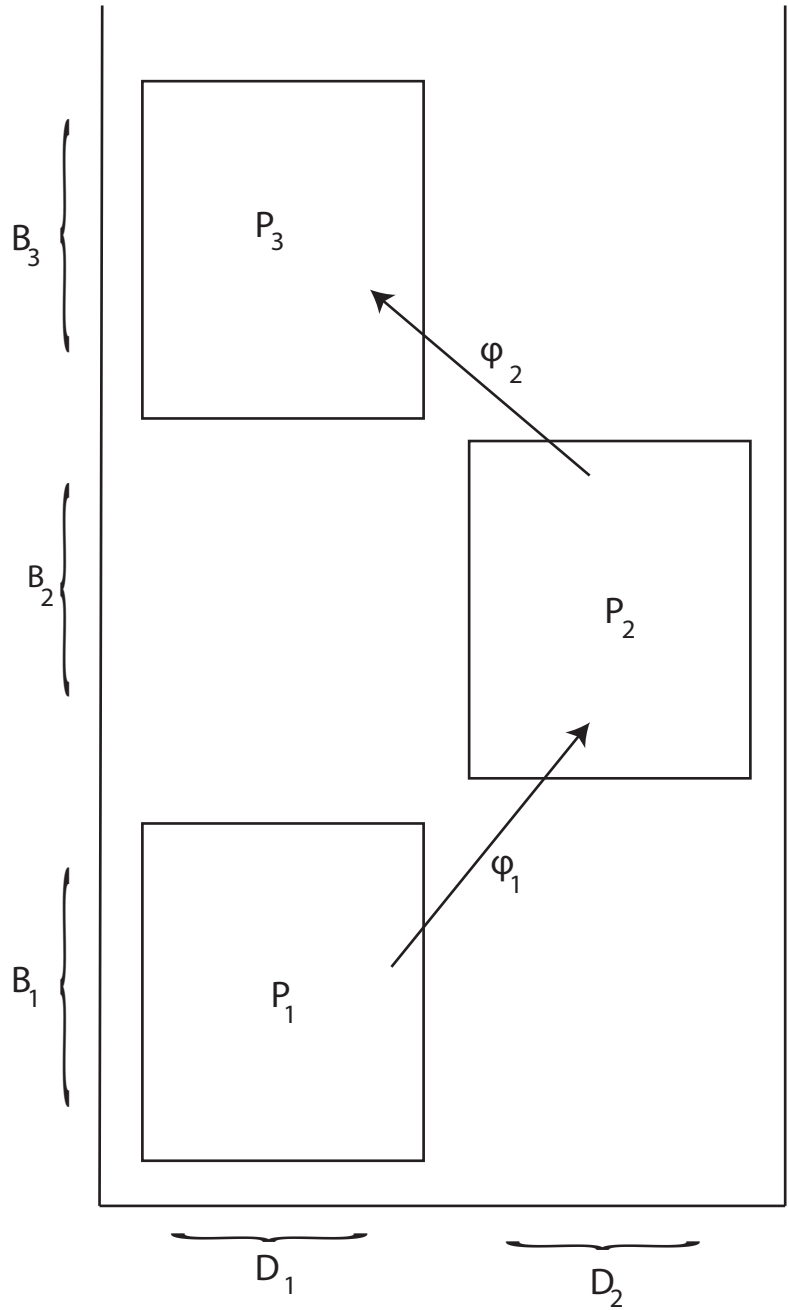}}
  \caption{Polydisks $P_i$ and diffeomorphisms $\phi_i$.}
  \label{thinfig1}
\end{figure}

The first few polydisks $P_i$ are illustrated in Figure \ref{thinfig1}, and the diffeomorphisms $\phi_i$ are described in the following lemma.

\begin{lemma} \label{polymove} There exist Hamiltonian diffeomorphisms $\phi_i$ of the $(z_2,z_3)$ plane, generated by time-independent, compactly supported, Hamiltonian functions $G_i$ of norm bounded by $\lambda + \epsilon$, such that $\phi_i(P_i) = P_{i+1}$.

Moreover, for $i$ odd and $0 \le t \le 1$, we have $$\phi^t_i(P_i) \subset D((1+t)(\lambda + \epsilon)) \times (A_i \cup A_{i+1}).$$ For $i$ even we have $$\phi^t_i(P_i) \subset D((2-t)(\lambda + \epsilon)) \times (A_i \cup A_{i+1}).$$
\end{lemma}

\begin{proof} We prove existence for the $\phi_i$ when $i$ is odd, the even case is similar.

There exists a Hamiltonian diffeomorphism $\psi_2$ of the $z_2$-plane of norm bounded by $\lambda + \epsilon$ such that the corresponding Hamiltonian is supported in $D(2(\lambda+\epsilon))$ and $\psi_2(D_1)=D_2$. It is not hard to produce a flow such that we have $\psi^t_2(D_1) \subset D((1+t)(\lambda + \epsilon))$. By abuse of notation let us also write $\psi_2$ for the Hamiltonian diffeomorphism $\psi_2 \times \mathrm{id.}$ of the $(z_2,z_3)$-plane.

Also, there exists a Hamiltonian diffeomorphism $\psi_3$ of the $z_3$-plane such that the corresponding Hamiltonian is supported in $A_i \cup A_{i+1}$ and $\psi_3(B_i)=B_{i+1}$. We may assume that both $\psi_2$ and $\psi_3$ are generated by time-independent Hamiltonian functions, say $H_2$ and $H_3$ respectively.

Let $\chi$ be a cut-off function with $\chi(x)=0$ when $x \le \lambda$ and $\chi(x)=1$ when $x \ge \lambda + \epsilon$. Then consider the Hamiltonian function $\chi(\pi|z_2|^2)H_3(z_3)$ on the $(z_2,z_3)$-plane. This generates a diffeomorphism $\tilde{\psi_3}$ which is the identity when restricted to $D(\lambda) \times \CC$ (where $\chi(\pi|z_2|^2)$ is identically $0$) but with $\tilde{\psi_3}(z_2,z_3)=(z_2, \psi_3(z_3))$ when $\pi|z_2|^2 \ge \lambda + \epsilon$.

Hence, since $D_1 = D(\lambda)$ and $D_2 \subset D(2(\lambda+\epsilon)) \setminus D(\lambda + \epsilon)$, the composition $\phi_i= \tilde{\psi_3} \circ \psi_2 \circ \tilde{\psi_3}^{-1}$ maps $P_i$ to $P_{i+1}$ and it remains to check that it has the required properties.

The diffeomorphism $\phi_i$ is the time-$1$ map of a Hamiltonian flow $\tilde{\psi_3} \circ \psi^t_2 \circ \tilde{\psi_3}^{-1}$. This flow is generated by a Hamiltonian function $G_i=H_2 \circ \tilde{\psi_3}^{-1}$, where $H_2$ is a Hamiltonian generating $\psi_2$, see for instance \cite{hoferzehnder} Proposition $1$, page $144$. Therefore we may assume $\phi_i$ has norm bounded by $\lambda + \epsilon$ (the bound on the norm of $H_2$).

Second, since $\psi^t_2(P_i) \subset D((1+t)(\lambda + \epsilon)) \times \CC$, and $\tilde{\psi_3}$ preserves $\pi|z_2|^2$ (since its Hamiltonian commutes with $\pi|z_2|^2$) we also have $\phi^t_i(P_i) \subset D((1+t)(\lambda + \epsilon)) \times \CC$ as required.
\end{proof}

With all of this in place, we can begin the construction.

\newpage

{\bf Step $1$.} {\it Repositioning the domain.}

Recall the symplectomorphism $\psi$ mapping $D(S)$ to a neighborhood of $V$. We apply the map $\phi_0 = \psi \times \mathrm{id.}$ to $E(S,1,T)$ and think of the image $F_0 = \phi_0(E(S,1,T))$ as fibered over (a neighborhood of) $V$ with projection $\pi_1:(z_1,z_2,z_3) \mapsto z_1$.

The fibers of $\pi_1$ are all contained in the bidisk $P(1,T)$. However, note that if $(z_1,z_2,z_3) \in E(S,1,T)$ and $\pi|z_1|^2 > \lambda$, then $\pi|z_2|^2 < 1-\frac{\lambda}{S} = \lambda$. Hence, since $\psi$ maps $D(\lambda)$ to the square $[0,1] \times [0,\lambda] $, the fibers of $\pi_1$ over points in the interval $[1,2N+1] \times \{0\}$ lie inside the smaller bidisk $P(\lambda,T)=P_1$.

{\bf Step $2$.} {\it Displacing fibers.}

Let $\chi_i$ be a cut-off function with $\chi_i(x)=0$ when $x \le 2i$ and $\chi_i(x)=1$ when $x \ge 2i+1$. Further, we may assume $0 \le \chi'_i(x) \le 1+\epsilon$. (Recall throughout that $\epsilon$ is any quantity which can be arbitrarily small.)

Let $G_i$ be the Hamiltonian function, with norm bounded by $\lambda + \epsilon$, generating the diffeomorphism $\phi_i$ described in Lemma \ref{polymove}. Then we can define $\sigma_i$ to be the Hamiltonian diffeomorphism generated by $\chi_i(\mathrm{Re}z_1) G_i(z_2,z_3)$ and set $$F_i = \sigma_i(\sigma_{i-1} \dots (\sigma_1(F_0)) \dots ).$$

The following is a key lemma. Recall that by a subset we really mean a subset of a small neighborhood.

\begin{lemma}\label{image} Let $W_i = [2i, 2i+1] \times [\min G_i, \max G_i] $ be a subset of the $z_1$-plane. Then $$\pi_1(F_N) \subset \left( [0,1] \times [0,\lambda] \right) \bigcup_{i=1}^N \left( [2i-1,2i]\times \{0\} \right) \bigcup_{i=1}^N W_i.$$

The fibers of $F_N$ over $[0,1] \times [0,\lambda]$ lie in $P(1,T)$.

The fibers of $F_N$ over $[2i-1,2i]\times \{0\}$ lie in $P_i$.

The fibers of $F_N$ over $W_i$ lie in $D(2(\lambda + \epsilon)) \times (A_i \cup A_{i+1})$.

More precisely, and as usual up to an error of order $\epsilon$, the fibers of $F_N$ over a point $z_1$ with $\mathrm{Re}(z_1) = 2i+t$ for some $0 \le t \le 1$ lie in $D((1+t)(\lambda + \epsilon)) \times (A_i \cup A_{i+1})$ if $i$ is odd and $D((2-t)(\lambda + \epsilon)) \times (A_i \cup A_{i+1})$ if $i$ is even.
\end{lemma}

Note that the union in the statement of the lemma is connected, because the $G_i$ can be chosen to have compact support.

\begin{proof} We will work by induction and show that $$\pi_1(F_k) \subset \left( [0,1] \times [0,\lambda] \right)  \bigcup_{i=1}^k \left( [2i-1,2i]\times \{0\} \right) \bigcup_{i=1}^k W_i \bigcup \left( [2k+1,2N+1] \times \{0\} \right)$$ and that the fibers are as described in the statement of the lemma over points $z_1$ with $\mathrm{Re}(z_1)\le 2k+1$. Furthermore, the fibers over $[2k+1,2N+1] \times \{0\}$ lie in $P_{k+1}$.

As noted at the end of Step $1$, this is the case for $F_0$ (where we have $\pi_1(F_0) \subset V$) so we assume that some $F_{k-1}$ has this property.

The diffeomorphism $\sigma_k$ is the identity on $\{\mathrm{Re}z_1 \le 2k\}$ (where $\chi_k =0$) and acts by $(z_1,z_2,z_3) \mapsto (z_1, \phi_k(z_2,z_3))$ on $\{\mathrm{Re}z_1 \ge 2k+1\}$. Therefore the fibers of $F_k$ over $[2k+1,2N+1] \times \{0\}$ lie in $P_{k+1}$ and it remains to check the image under $\sigma_k$ of points with $z_1=2k+t$ for some $0 \le t \le 1$.

Using the complex notation above, these points flow according to the time independent vector field $$X_k = i(\chi_k(\mathrm{Re}z_1) \nabla G_k(z_2,z_3) + \nabla \chi_k(\mathrm{Re}z_1) G_k(z_2,z_3)).$$ The flow preserves $\mathrm{Re}z_1$ and when $\mathrm{Re}z_1=2k+t$ is given by $X_k = i(\chi_k(2k+t) \nabla G_k(z_2,z_3) + \nabla \chi_k(2k+t) G_k(z_2,z_3))$. We have that $\chi_k(k+t)$ is roughly equal to $t$ and $\nabla \chi_k(k+t)$ is bounded by $1+\epsilon$. Therefore the component of the flow in the $z_1$-plane is parallel to the imaginary axis and has velocity bounded above and below by $\max G_k$ and $\min G_k$ respectively, up to terms of order $\epsilon$. The component in the $(z_2,z_3)$-plane is roughly $it\nabla G_k$ and since $G_k$ is time-independent the time-$1$ flow is equivalent to flowing by $i\nabla G_k$ for time $t$. By Lemma \ref{polymove} this has image in $D((1+t)(\lambda + \epsilon)) \times (A_k \cup A_{k+1})$ if $k$ is odd and  $D((2-t)(\lambda + \epsilon)) \times (A_k \cup A_{k+1})$ for $k$ even.
\end{proof}

{\bf Step $3$.} {\it Folding.}

Let $\tau:\CC \to \CC$ be a symplectic immersion of a neighborhood of $\pi_1(F_N)$ in the $z_1$-plane with the following properties. The immersion $\tau$ restricts to an embedding on the square $W_0 = [0,1] \times [0,\lambda]$, on each $W_i$ for $i \ge 1$, and on each interval $[2i-1,2i]\times \{0\}$. Moreover, we assume that the images of the intervals $[2i-1,2i]\times \{0\}$ are disjoint from the images of the $W_i$, and that the images of the $W_i$ for $i$ odd are disjoint from the images of the $W_i$ for $i$ even. As the $W_i$ have area roughly $\lambda$ (the norm of the $G_i$), we may assume that $\tau$ maps $\pi_1(F_N)$ to a neighborhood of $D(2\lambda)$. In fact, we will assume that $\tau$ maps the $W_i$ for $i$ even into a neighborhood of $D(\lambda)$ and the $W_i$ for $i$ odd close to $D(2\lambda) \setminus D(\lambda)$. Such an immersion is illustrated in Figure \ref{thinfig2}. Moreover, using \cite{schl}, Lemma $3.1.5$, we can further assume that points $z_1 \in W_i$ for $i$ odd with $\mathrm{Re}z_1=2i+t$ are mapped close to $D((2-t)\lambda)$. (Comparing with Lemma \ref{image}, this means the points whose fibers lie in the largest bidisks are mapped closer to the center of $D(2\lambda)$.)

\begin{figure}
 \centerline{\includegraphics{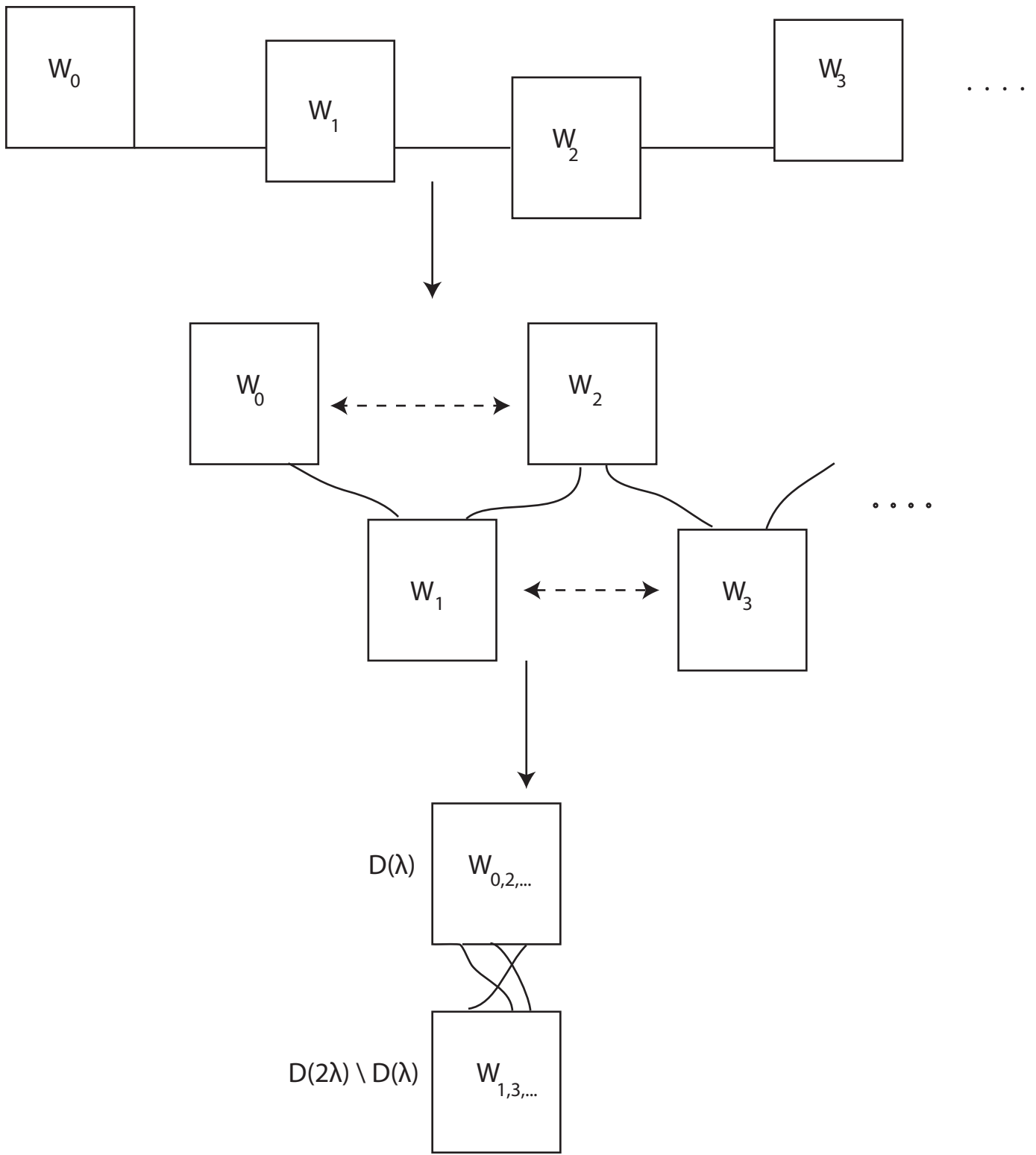}}
  \caption{The immersion $\tau$ of the $z_1$-plane.}
  \label{thinfig2}
\end{figure}

Finally consider the symplectic immersion $\phi=\tau \times \mathrm{id}$. We claim that this restricts to give an embedding of $F_N$ and furthermore satisfies $\phi(F_N) \subset (B^4(3\lambda) \cap P(2\lambda, 2\lambda)) \times \CC$. This will complete our construction.

Firstly, $\phi|_{F_N}$ is an embedding since by Lemma \ref{image}, if $i \neq j$ are either both even or both odd then the fibers of $F_N$ over $W_i$ and the fibers over $W_j$ are disjoint. Also, the fibers over the interval $[2i-1,2i]\times \{0\}$ lie in $P_i$ and so the fibers over different intervals are also disjoint.

It remains to check the image. Note that the coordinate $z_2 \in D(2\lambda)$ (or really a small neighborhood) for all points in our images. We use here that $\lambda \ge \frac{1}{2}$ to justify this for the images of points in the fibers over $W_0$. Therefore if $(z_1,z_2,z_3) \in \phi(F_N)$ and $z_1 \in D(\lambda)$ we have $(z_1,z_2) \in B^4(3\lambda) \cap P(2\lambda, 2\lambda)$ as required. Suppose then that $\pi|z_1|^2 = (1+t)\lambda$ for some $0 \le t \le 1$. Then by the construction of $\tau$ we may assume that $(z_1,z_2,z_3)$ is the image of a point $(z'_1,z_2,z_3)$ in a $W_i$ for $i$ odd with $\mathrm{Re}z'_1 \le 2i+1-t$. Hence by Lemma \ref{image} we have $z_2 \in D((2-t)\lambda) \subset D(2\lambda)$. Thus $\pi|z_1|^2 + \pi|z_2|^2 \le (1+t)\lambda + (2-t)\lambda = 3\lambda$ and again we see that $(z_1,z_2) \in B^4(3\lambda) \cap P(2\lambda, 2\lambda)$ as required.

\end{section}

\begin{section}{Embedding obstructions}\label{obstr}

In this section we show that our embeddings are optimal. This is an almost immediate consequence of Propositions $3.4$ and $3.14$ in \cite{hindker}, but to state the conclusion we need to recall some notation. The open ball $\mathring{B}^4(R)$ can be compactified by the complex projective plane $\CC P^2(R)$ equipped with a symplectic form where lines have area $R$. Then we study symplectic embeddings of an ellipsoid $$E \rightarrow \mathring{B}^4(R) \times \CC^{n-2} \subset \CC P^2(R) \times \CC^{n-2}.$$
Identifying $E$ with its image, the symplectic manifold $X = \left( \CC P^2(R) \times \CC^{n-2} \right) \setminus E$ admits a tame almost-complex structure, which, in a neighborhood of $\partial E$, is biholomorphic to $\partial E \times (-\infty, 0)$ with a translation invariant almost-complex structure $J$, see \cite{EGH} and \cite{hindker}, section $3.1$. Here $\partial E$ carries a contact form $\alpha$ coming from the standard Liouville form on $\CC^n$. This induces a contact structure $\xi=\{\alpha=0\}$ on $\partial E$ and a corresponding Reeb vector field $v$. Our almost-complex structure structure can be chosen so that $J$ preserves $\xi$ and $J(v)=-\frac{\partial}{\partial t}$, where $t$ is the coordinate on $(-\infty, 0)$. We denote by $\gamma$ the closed Reeb orbit $\{z_j=0, j \ge 2\} \cap \partial E$. Assuming $E=E(a_1, \dots ,a_n)$ with $a_j \ge a_1$ for all $j$, this is the orbit of shortest action. Let $r \gamma$ be the $r$ times cover of $\gamma$.

Given all of this, we can study finite energy $J$-holomorphic curves in $X$ asymptotic to closed Reeb orbits on $\partial E$. For the foundations of finite energy curves see \cite{hofa}, \cite{hofi}, \cite{hoff}.
The key existence result for finite energy curves which gives our embedding obstructions is the following, see \cite{hindker}, section $3.4$.

\begin{theorem}\label{fep} (Hind-Kerman, \cite{hindker}, Propositions $3.4$ and $3.14$) Let $E \subset \mathring{B}^4(R) \times \CC^{n-2}$ be the image of an ellipsoid $E(1,S_2, \dots ,S_n)$ with $S_j \ge 3d-1$ for all $j$. Then there exists a finite energy plane of degree $d$ in $X$ asymptotic to $(3d-1)\gamma$.
\end{theorem}

The degree of a finite energy curve in $X$ can be taken to be its intersection number with $\CC P^1(\infty) \times \CC^{n-2}$, where $\CC P^1(\infty)$ is the line at infinity in $\CC P^2(R)$ and the intersection number is just the number of intersections counted with multiplicity.

Finite energy curves have positive symplectic area. Since lines in $\CC P^2(R)$ have area $R$ and $\gamma$ has action $1$, a computation using Stokes' Theorem of the area of the planes given by Theorem \ref{fep} therefore implies $$dR - (3d-1) >0$$ or $$R > 3 - \frac{1}{d}.$$ If $a_2 = 3d-1$ then $\frac{3a_2}{a_2+1} = 3 - \frac{1}{d}$ and so this bound on $R$ implies that the corresponding embedding of Theorem \ref{main} is optimal.

There is a similar construction of finite energy planes into products which gives obstructions to embeddings of an ellipsoid $E$ into a product $\mathring{P}(R_1, R_2) \times \CC^{n-2}$. Now we compactify $\mathring{P}(R_1, R_2)$ to the product of spheres $S^2(R_1) \times S^2(R_2)$, where $S^2(R)$ denotes a sphere of area $R$. Setting $Y = \left( S^2(R_1) \times S^2(R_2) \times \CC^{n-2} \right) \setminus E$, we can find a tame almost-complex structure as before and study finite energy holomorphic curves in $Y$. The relevant existence result is then as follows.

\begin{theorem}\label{fep2} Let $E \subset \mathring{P}(R_1,R_2) \times \CC^{n-2}$ be the image of an ellipsoid $E(1,S_2, \dots ,S_n)$ with $S_j \ge 2d+1$ for all $j$ and $R_1 \le R_2$. Then there exists a finite energy plane of bidegree $(d,1)$ in $Y$ asymptotic to $(2d+1)\gamma$.
\end{theorem}

We say that the bidegree of a finite energy curve in $Y$ is $(k,l)$ if its intersection number with $\infty \times S^2(R_2) \times \CC^{n-2}$ is $k$ and its intersection number with $S^2(R_1) \times \infty \times \CC^{n-2}$ is $l$.

Now applying Stokes' Theorem to the planes in Theorem \ref{fep2} gives $$dR_1 + R_2 > 2d+1$$ and setting $R_1=R_2$ this shows that the embedding of Theorem \ref{main2} is optimal as claimed when $a_2 = 2d+1$.

\end{section}


\begin{thebibliography} {99}





\bibitem{busehind} O. Buse and R. Hind, Symplectic embedding of ellipsoids in dimension greater than four, {\em Geom. Top.}, 15 (2011), 2091--2110.


\bibitem{choi} K. Choi, D. Cristofaro-Gardiner, D. Frenkel, M. Hutchings and V. G. B. Ramos, Symplectic embeddings into four-dimensional concave toric domains, {J. Topol.}, to appear.



\bibitem{ekehof} I. Ekeland and H. Hofer, Symplectic topology and
Hamiltonian dynamics II, {\it Math. Z.}, 203 (1990), 553--567.



\bibitem{EGH} Y. Eliashberg, A. Givental and H. Hofer, Introduction to symplectic field theory, GAFA 2000 (Tel Aviv, 1999), {\em Geom. Funct. Anal.}, 2000, Special Volume, Part II, 560--673.



\bibitem{frmu} D. Frenkel and D. M\"{u}ller, Symplectic embeddings of 4-dimensional ellipsoids into cubes, arXiv:1210.2266.

\bibitem{gromov} M. Gromov, Pseudo-holomorphic curves in symplectic manifolds, {\it Inv. Math.}, 82 (1985), 307--347.

\bibitem{guth} L. Guth, Symplectic embeddings of polydisks, {\it Inv.  Math}, \textbf{172} (2008), 477--489.





\bibitem{hindker} R. Hind and E. Kerman, New obstructions to symplectic embeddings, {\it Inv. Math.}, 196 (2014), 383--452.


\bibitem{hofer} H. Hofer, Symplectic capacities. In: Geometry of Low-dimensional Manifolds 2
(Durham, 1989). Lond. Math. Soc. Lect. Note Ser., vol. 151, pp. 15–34. Cambridge
Univ. Press, Cambridge (1990).



\bibitem{hofa} H. Hofer, K. Wysocki and E. Zehnder, Properties of pseudoholomorphic curves in symplectisations I: Asymptotics, {\it Ann. Inst. H. Poincar\'{e} Anal. Non Lineaire}, 13 (1996) , 337--379.

\bibitem{hofi} H. Hofer, K. Wysocki and E. Zehnder, Properties of pseudoholomorphic curves in symplectisations II: Embedding controls and algebraic invariants, {\it Geom. Funct. Anal.}, 5 (1995),  337--379.

\bibitem{hoff} H. Hofer, K. Wysocki and E. Zehnder, Properties of pseudoholomorphic curves in symplectisations III: Fredholm theory, {\it Topics in nonlinear analysis}, 381-475, Prog. Nonlinear Differential Equations Appl., 35, Birkh\"{a}user, Basel, 1999.

\bibitem{hoferzehnder} H. Hofer and E. Zehnder, Symplectic invariants and Hamiltonian dynamics,  Birkh\"{a}user, Basel, 1994.




\bibitem{hutchings} M. Hutchings, Quantitative embedded contact homology, {\it J. Diffl. Geom.} 88 (2011), 231-–266.

\bibitem{hutchings2} M. Hutchings, Recent progress on symplectic embedding problems in four dimensions, {\it Proc. Natl. Acad. Sci. USA}, 108 (2011), 8093-–8099.






\bibitem{mcdsch} D. McDuff and F. Schlenk, The embedding capacity of $4$-dimensional symplectic ellipsoids, {\it Ann. of Math.}, 175 (2012), 1191--1282.

\bibitem{mcduff} D. McDuff, Symplectic embeddings of $4$-dimensional ellipsoids, {\it J. Topol.}, 2 (2009), 1--22.

\bibitem{mcd2} D. McDuff,
The Hofer conjecture on embedding symplectic ellipsoids, {\it J. Diffl. Geom.}, 88 (2011), 519-–532.



\bibitem{pelngo1} A. Pelayo and S. V. Ng\d{o}c, The Hofer question on intermediate symplectic capacities, preprint, arXiv:1210.1537.

\bibitem{pelngo2} A. Pelayo and S. V. Ng\d{o}c, Sharp symplectic embeddings of cylinders, preprint, arXiv:1304.5250.



\bibitem{schl} F. Schlenk, Embedding problems in symplectic geometry
De Gruyter Expositions in Mathematics 40. Walter de Gruyter Verlag, Berlin. 2005.







\end{thebibliography}
\end{document}